\documentclass[12pt,reqno]{amsart}
\usepackage{xcolor}
\usepackage{amssymb}
\usepackage{tikz-cd}
\usepackage{hyperref}

\makeatletter
\@namedef{subjclassname@2020}{%
\textup{2020} Mathematics Subject Classification}
\makeatother
\newtheorem{thm}{Theorem}[section]
\newtheorem{lem}[thm]{Lemma}
\newtheorem{prop}[thm]{Proposition}

\theoremstyle{definition}
\newtheorem{defi}[thm]{Definition}
\newtheorem{xrem}[thm]{Remark}
\newtheorem{exm}[thm]{Example}
\newtheorem{question}[thm]{Question}

\DeclareMathOperator{\mult}{{mult}}
\DeclareMathOperator{\rank}{{rank}}

\DeclareMathOperator{\End}{{End}}

\DeclareMathOperator{\Bl}{{Bl}}

\begin{document}
\baselineskip=16pt

\title[Some results on Seshadri constants of vector bundles]{Some results on Seshadri constants of vector bundles}

\author[I. Biswas]{Indranil Biswas}

\address{Department of Mathematics, Shiv Nadar University, NH91, Tehsil
Dadri, Greater Noida, Uttar Pradesh 201314, India}

\email{indranil.biswas@snu.edu.in, indranil29@gmail.com}

\author[K. Hanumanthu]{Krishna Hanumanthu}

\address{Chennai Mathematical Institute, H1 SIPCOT IT Park, Siruseri, Kelambakkam 603103,
India}

\email{krishna@cmi.ac.in}

\author[S. Misra]{Snehajit Misra}

\address{Chennai Mathematical Institute, H1 SIPCOT IT Park, Siruseri, Kelambakkam 603103, 
India}

\email{snehajitm@cmi.ac.in}



\date{4 August, 2023}

\begin{abstract}
We study Seshadri constants of certain ample vector bundles on projective varieties. Our main 
motivation is the following question:\, Under what conditions are the Seshadri constants of ample vector 
bundles at least 1 at all points of the variety. We exhibit some conditions under which this 
question has an affirmative answer. We primarily consider ample bundles on projective spaces 
and Hirzebruch surfaces. We also show that Seshadri constants of ample vector bundles can be 
arbitrarily small.
\end{abstract}

\maketitle

\section{Introduction}\label{sect1}

Seshadri constants of line bundles on projective varieties were defined by Demailly \cite{Dem90} 
in 1990 and since then they have been an active area of research around positivity in algebraic 
geometry. This notion has been generalized to vector bundles by Hacon \cite{Hac} and recently
Seshadri constants for vector bundles have
been studied in depth by Fulger and Murayama \cite{FM21}. While a lot is known for Seshadri 
constants of line bundles, especially on surfaces, not much is known about
Seshadri constants of vector bundles on varieties of
dimension at least two. For instance we do not know how to compute Seshadri constants for arbitrary 
vector bundles on the projective plane $\mathbb{P}^2$. For vector bundles on curves, the picture 
is much clearer.

The usual notions of positivity of line bundles have been generalized to vector bundles, usually 
by considering the tautological line bundle on the associate projective bundle. For instance a 
vector bundle $E$ on a variety $X$ is called ample if the tautological line bundle 
$\mathcal{O}_{\mathbb{P}(E)}(1)$ on $\mathbb{P}(E)$ is ample. It is natural to define and study 
Seshadri constants of vector bundles following this approach.

Let $X$ be a projective variety and $E$ a nef
vector bundle on $X$.
The Seshadri constant $\varepsilon(X,\,E,\,x)$ of $E$ at a closed point $x \,\in\, X$ is
defined by Hacon \cite{Hac}.

If $L$ is an ample and globally generated line bundle on a projective variety $X$, then it is
easy to show that $\varepsilon(X,\,L,\,x)\,\ge\, 1$ for all $x\,\in\, X$. On the other hand, 
Miranda has shown that the Seshadri constants can be arbitrary small for ample line bundles 
(which are, necessarily, not globally generated). However in a fundamental paper, Ein and 
Lazarsfeld, \cite{EL93}, showed that
$$
\varepsilon(X,\,L,\,x)\,\ge \,1
$$
for \textit{very general points} $x$ in a smooth projective irreducible surface $X$ when $L$ is ample.

The above behaviour does not carry over to the case of vector bundles of higher rank. For vector 
bundles over curves, one still has some good bounds.
For an ample vector bundle $E$ on a smooth curve $X$, Hacon
\cite[Theorem 3.1]{Hac} proved that 
\begin{eqnarray}\label{hacon}
\varepsilon(X,\,E,\,x)\, =\, \mu_{\min}(E)
\end{eqnarray}
for all $x\,\in\, X$.
This implies that $$\varepsilon(X,E,x)\,\,\ge\,\, \frac{1}{\text{rank}(E)}$$ for all $x\,\in\, X$.
Moreover, for every possible rank there are examples where this inequality is actually an
equality. On the 
other hand, $\varepsilon(X,\,E,\,x)\,\ge \,1$ for any ample and \textit{globally generated} vector 
bundle $E$ on a smooth projective curve $X$ and all points $x \,\in\, X$ \cite[Proposition 3.2]{Hac}.

Much less is known when the dimension of the base is at least two. Hacon \cite{Hac96} gave examples of ample and 
globally generated vector bundles $E$ on K3 surfaces for which $\varepsilon(X,E,\,x) \,<\, 1$ for all
points $x$. Hence the Ein--Lazarsfeld inequality
\begin{equation}\label{iel}
\varepsilon(X,\,E,\,x)\,\ge\, 1
\end{equation}
does not hold, in 
general, for ample and globally generated vector bundles of rank at least two. It however holds 
for \textit{very ample} vector bundles; see \cite{BSS96}.

So it is reasonable to ask when the Ein--Lazarsfeld inequality \eqref{iel} holds for ample vector bundles of 
arbitrary rank on arbitrary projective varieties.

Most of the results in this paper are motivated by the above question. There does not seem to be a 
general method to approach this question, so we focus on specific cases and consider vector 
bundles satisfying specific conditions.

In Section \ref{sect2}, we recall some preliminaries that are used in later sections. In 
Section \ref{sect3}, we prove \eqref{iel} for vector bundles satisfying a certain 
hypothesis and give several examples. In Section \ref{sect4}, we first show that the Seshadri 
constants of ample vector bundles can be arbitrarily small. We also give lower bounds for the
Seshadri constants for vector bundles satisfying some conditions. In Section \ref{sect5}, we 
consider some vector bundles on Hirzebruch surfaces.

\section{Preliminaries}\label{sect2}

We work throughout over the field of complex numbers. All the varieties are assumed to be irreducible and reduced.

\subsection{Seshadri constants of vector bundles} Let $X$ be an irreducible complex projective variety and let $E$ be a nef vector
bundle on $X$. We fix a closed point $x\in X$.

Let us consider the following Cartesian diagram:
\begin{center}
\begin{tikzcd}
\mathbb{P}(E)\times_X \widetilde{X}_x = \mathbb{P}\bigl(\rho_x^{\star}(E)\bigr) \arrow[r, "\widetilde{\rho_x}"] \arrow[d, "\widetilde{\pi}"]
 & \mathbb{P}(E) \arrow[d,"\pi"]\\
 \widetilde{X}_x \arrow[r, "\rho_x" ]
 & X
\end{tikzcd}
\end{center}
Here $\pi\,:\, \mathbb{P}(E)\,\longrightarrow\, X$ is the projective bundle over $X$ associated to
$E$ and $$\rho_x \,:\, \widetilde{X}_x \,= \,\Bl_x(X)\,\longrightarrow\, X$$
is the blow up of $X$ at $x$.
Let $\widetilde{\xi}_x$ be the numerical equivalence class of the tautological bundle
$\mathcal{O}_{\mathbb{P}(\rho^{\star}_xE)}(1)$, and $\widetilde{E}_x\,:=\,
\widetilde{\rho_x}^{-1}(F_x)$, where $F_x$ is the class of the fiber of the map $\pi$ over
the point $x$.
Then the \textit{Seshadri constant of $E$ at} $x\,\in\, X$ is defined as
\begin{align*}
\varepsilon(X,\,E,\,x) \,:=\, \sup\,\Bigl\{ \lambda \,\in\, \mathbb{R}_{>0}
\,\mid\,\, \widetilde{\xi}_x - \lambda\widetilde{E}_x\,\, \text{ is nef} \Bigr\}.
\end{align*}
In \cite{Hac} it is shown that if $E$ is an ample vector bundle on a smooth irreducible projective
curve $C$, then $\varepsilon(C,\,E,\,x) = \mu_{\min}(E)$ for all $x\,\in\, C$, where
$\mu_{\min}(E)$ denotes the minimal slope among all quotient bundles of $E$.
There is an equivalent definition of Seshadri constants for vector bundles due to
Fulger and Murayama in \cite{FM21}. Let $\mathcal{C}_{\pi, x}$ denote the set of all irreducible
curves $C\,\subset\, \mathbb{P}(E)$ that intersect $\pi^{-1}(x)$, but not contained in
$\pi^{-1}(x)$. Note that the curves in $\mathcal{C}_{\pi,x}$ are precisely the irreducible curves $C$ in $\mathbb{P}(E)$ such that $\mult_x\pi_*C > 0$. Then the Seshadri constants of $E$ at $x$
has the following expression:
$$\varepsilon(X,\,E,\,x)\, =\, \inf\limits_{C\in \mathcal{C}_{\pi,x}}\,
\left\{\frac{\xi\cdot C}{\mult_x\pi_*C}\right\},$$
where the infimum is taken over the set $\mathcal{C}_{\pi,x}$.

Seshadri constants can also be computed by restriction to curves \cite[Corollary 3.21]{FM21}.
This description, which will be used here, is recalled below.

Let $C\,\subset\, X$ be an irreducible curve in $X$, possibly singular. Take its
normalization $\alpha\, :\, C'\, \longrightarrow\, C$. Let $\nu \,:\, C'
\,\longrightarrow \, X$ denote the composition of $\alpha$
with the inclusion map $C\,\hookrightarrow\, X$. Then
\begin{eqnarray}\label{curves}
\varepsilon(X,\,E,\,x)\,=\, \inf\limits_{x\in C \subset X}\,
\left\{\frac{\mu_{\min}(\nu^{\star}E)}{\mult_xC}\right\},
\end{eqnarray}
where the infimum is taken over all irreducible curves $C\,\subset\, X$ passing through $x\,\in\,X$
(see \cite[Corollary 3.21]{FM21}).
\subsection{Semistability and Harder-Narasimhan filtration}
Let $X$ be a smooth complex projective variety of dimension $n$ with a fixed ample line bundle 
$H$ on it. For a non-zero torsion-free coherent sheaf $\mathcal{G}$ of rank $r$ on $X$, its degree with
respect to $H$, denoted by $\deg_H(\mathcal{G})$, is defined as follows:
$$\deg_H(\mathcal{G}) \,=\,c_1(\mathcal{G})\cdot H^{n-1},$$
and the $H$-slope of $\mathcal{G}$ is defined as 
\begin{align*}
\mu_H(\mathcal{G}) \,:=\, \frac{c_1(\mathcal{G})\cdot H^{n-1}}{r} \,\in\, \mathbb{Q}.
\end{align*}
A vector bundle $E$ on $X$ is said to be $H$-semistable (respectively, $H$-stable) if $\mu_H(\mathcal{G}) \,\leq\, \mu_H(E)$
(respectively, $\mu_H(\mathcal{G}) \,<\, \mu_H(E)$) for 
all subsheaves $\mathcal{G}\, \subsetneq\, E$ such that $E/\mathcal{G}$ is torsion free. 
A vector bundle $E$ on $X$ is called $H$-unstable if it is not $H$-semistable. For every vector bundle $E$ on $X$, there is a unique
positive integer $l$ and a unique filtration
\begin{align*}
 0 \,=\, E_l \,\subsetneq\, E_{l-1} \,\subsetneq\, E_{l-2} \,\subsetneq\,\cdots\,\subsetneq \,E_{1} \,\subsetneq\, E_0 \,=\, E
\end{align*}
of subsheaves of $E$, called the \textit{Harder-Narasimhan filtration} of $E$, such that $E_i/E_{i+1}$ is an $H$-semistable torsion-free sheaf
for each $i \,\in\, \{0,\,1,\,2,\,\cdots,\,l-1\}$
and $$\mu_H\bigl(E_{l-1}/E_{l}\bigr) \,>\, \mu_H\bigl(E_{l-2}/E_{l-1}\bigr) \,>\,\cdots\,>\, \mu_H\bigl(E_{0}/E_{1}\bigr).$$
We define $Q_1\,:=\, E_{0}/E_{1}$ and $\mu_{\min}(E) \,:=\, \mu_H(Q_1) = \mu_H\bigl(E_{0}/E_{1}\bigr)$; also
define $\mu_{\max}(E) \,:=\, \mu_H (E_{l-1}).$

For a vector bundle $E$ of rank $r$ over $X$, the element $c_2\bigl(\End(E)\bigr) \in H^4(X,\mathbb{Q})$ is called the 
\textit{discriminant} of $E$. We recall the following result about semistable bundle with vanishing discriminant which we will use in 
Section \ref{sect4}:

\begin{thm}[{\cite[Theorem 1.2]{BB08}}]\label{thm1.1}
Let $E$ be a vector bundle of rank $r$ on a smooth complex projective variety $X$, and
let $\pi : \mathbb{P}(E) \longrightarrow X$ be the projection. Then the following two are equivalent:
\begin{enumerate}
\item $E$ is semistable and $c_2\bigl(\End(E)\bigr) = 0$.

\item For every pair of the form $(\phi,C)$, where $C$ is a smooth projective curve and
$\phi : C \longrightarrow X$ is a non-constant morphism, $\phi^*(E)$ is semistable.
\end{enumerate}
\end{thm}

Since the semistability of a vector bundle on a smooth curve does not depend on the fixed ample line bundle on $C$, the condition (2) in 
the above Theorem \ref{thm1.1} implies that the semistability of a vector bundle $E$ with $c_2\bigl(\End(E)\bigr) = 0$ is independent of 
the fixed ample bundle $H$. We have the following lemma which is an easy applications of Theorem \ref{thm1.1}.

\begin{lem}\label{lem2.2}
 Let $\psi : X \longrightarrow Y$ be a morphism between two smooth complex projective varieties and let $E$ be a semistable bundle on $Y$ with $c_2\bigl(\End(E)\bigr) = 0$. Then the twisted pullback bundle $\psi^*(E)\otimes L$ is also semistable with $c_2\bigl(\End(\psi^*(E)\otimes L)\bigr) = 0$ for any line bundle $L$ on $X$.
\end{lem}

The following lemma will be used in Section \ref{sect5}.

\begin{lem}[{\cite[Lemma 3.31]{FM21}}]\label{mu-min}
Let $$0\,\longrightarrow\, E'\,{\longrightarrow}\, E\, \longrightarrow \,E'' \,\longrightarrow\, 0$$ be a short exact sequence of
vector bundles on a smooth irreducible projective curve $C$. Then
$$\mu_{\min}(E) \,\geq\, \min\,\bigl\{\mu_{\min}(E'),\, \mu_{\min}(E'')\bigr\}.$$
\end{lem}

\section{Ein-Lazarsfeld inequality for vector bundles}\label{sect3}

In this section, we give one instance of vector bundles for which the Ein--Lazarsfeld inequality
\eqref{iel} holds for arbitrary points. We give several 
examples to illustrate this result.

\begin{prop}\label{prop1}
Let $X$ be a smooth complex irreducible projective variety and $L$ an ample and globally generated line bundle on $X$. Let $E$ be an
ample vector bundle on $X$ such that $E\otimes L^{-1}$ is nef. Then 
$$\varepsilon(X,E,\,x) \,\geq\, 1$$ for all $x\,\in\, X$.
\end{prop}
 
\begin{proof}
Using \cite[Lemma 3.28]{FM21}, we have for every point $x\,\in\, X$,
$$\varepsilon(X,E,\,x)\,=\, \varepsilon\bigl(X,(E\otimes L^{-1})\otimes L,\,x\bigr)
\,\geq\, \varepsilon(X,E\otimes L^{-1},\,x) + \varepsilon(X,L,\,x) \,\geq\, 1.$$
\end{proof}
 We ask the following question --- which is a
variation of our motivating question --- in the particularly simple case of $X \,= \,\mathbb{P}^n$.

\begin{question}\label{qn1}
Is it true that for an ample vector bundle $E$ on $\mathbb{P}^n$,
$$\varepsilon(\mathbb{P}^n,\,E,x) \,\geq\, 1$$
for all $x\,\in\, \mathbb{P}^n$?
\end{question}

Question \ref{qn1} has a positive answer if $E$ is a direct sum of line bundles. Note that Hartshorne's
conjecture on vector bundles says that any rank two vector bundle over a projective space $\mathbb{P}^n$ with $n\geq7$ splits into the direct sum of two line bundles; see \cite[Conjecture 6.3]{Har74}. In general, Question \ref{qn1} is difficult to answer for non-split vector bundles even for rank two.

\begin{defi}
A vector bundle $E$ on a smooth projective variety $X$ is called \textit{Fano} if $\mathbb{P}(E)$ is a Fano variety,
meaning the anti-canonical line bundle $-K_{\mathbb{P}(E)}$ is ample.
\end{defi}

There are classifications of Fano bundles of rank two on smooth surfaces and on 3-folds in \cite{SW290} and \cite{SW190}
respectively. We will now consider some examples of non-split Fano bundles of rank two on projective varieties.

\begin{exm}\label{ex3.4}
Let $E$ be the rank two vector bundle on $\mathbb{P}^2$ which fits in the exact sequence
$$0\,\longrightarrow\, \mathcal{O}_{\mathbb{P}^2}\,\longrightarrow\, E \,\longrightarrow\, \mathcal{I}_{t}
\,\longrightarrow\, 0,$$
where $\mathcal{I}_t$ denotes the ideal sheaf of a point $t\in\mathbb{P}^2$. Note that $E$ is a semistable bundle
but it is not stable. We will show that $E(1)$ is nef, but $E(1)$ is not ample.

Note that $\mathcal{I}_t(1)$ is globally generated and hence nef. Thus being the extension of two nef sheaves, $E(1)$ is also nef.
Also $\mathcal{I}_t(1)$ restricts trivially to any line in $\mathbb{P}^2$ through $t$. Therefore $E(1)$ is not ample.
Consequently, Proposition \ref{prop1} says that
$$\varepsilon\bigl(\mathbb{P}^2,E(k),\,x\bigr)\,\geq \,1$$ for all $x\,\in\, \mathbb{P}^2$ and for all $k\,\geq\, 2$.
\end{exm}

In Example \ref{exm3.9}, we generalize Example \ref{ex3.4} to the ideal sheaf of five points in general position in $\mathbb{P}^2$. 

\begin{exm}\label{exm2}
Consider the rank two vector bundle $E$ on $\mathbb{P}^2$ fitting in the exact sequence
$$0\,\longrightarrow\, \mathcal{O}_{\mathbb{P}^2}(-1)^2\,\longrightarrow \,\mathcal{O}_{\mathbb{P}^2}^{\oplus4}
\,\longrightarrow\, E(1) \,\longrightarrow\, 0.$$
We note that $E$ is stable. The vector bundle $E(1)$ is globally generated and hence nef. But $E(1)$ is not ample as
there are lines $L$ in $\mathbb{P}^2$ such that $E(1)\big\vert_L \,=\, \mathcal{O}_{L}\oplus \mathcal{O}_{L}(2)$; see 
\cite[Page 298]{SW290}. 

Thus $E(k)$ is ample if and only if $k \ge 2$ and 
$$\varepsilon\bigl(\mathbb{P}^2,E(k),\,x\bigr) \,\geq\, 1$$ for all $x\,\in \,\mathbb{P}^2$ and all $k\,\geq\, 2$.
\end{exm}

\begin{exm}\label{exm3}
Consider the rank two vector bundle $E$ on $\mathbb{P}^2$ defined by
the exact sequence
$$0\,\longrightarrow\, \mathcal{O}_{\mathbb{P}^2}(-2)\,\longrightarrow\,
\mathcal{O}_{\mathbb{P}^2}^{\oplus 3} \,\longrightarrow\, E(1) \,\longrightarrow\, 0.$$
In this case $E$ is stable.
The vector bundle $E(1)$ is globally generated, see \cite[Proposition 2.1, Proposition 2.6]{SW290}. 
However $E(1)$ is not ample: One can show that $E(1)$ sits in an exact sequence of the form
$$ 0\to \mathcal{O}_{\mathbb{P}^2} \to E(1) \to \mathcal{I}(2) \to 0,$$ where
$\mathcal{I}$ is the ideal sheaf of four points in general position in $\mathbb{P}^2$. 
If $L$ is a line through two of these four points, then one can then check that
$E(1)\big\vert_L \,=\, \mathcal{O}_{L}\oplus \mathcal{O}_{L}(2)$; see also \cite[5.2]{B2}. So $E(1)$ is not ample. 

Thus $E(k)$ is ample if and only if $k \ge 2$ and 
$$\varepsilon\bigl(\mathbb{P}^2,E(k),\,x\bigr)\,\geq\, 1$$ for all $x\,\in\, \mathbb{P}^2$ and all $k\,\geq\, 2$.
\end{exm}

\begin{exm}\label{exm4}
Consider the null correlation bundle $\mathcal{N}$ on $\mathbb{P}^3$. We have the exact sequence
$$0\,\longrightarrow \,\mathcal{N} \,\longrightarrow\, T_{\mathbb{P}^3}(-1)\,\longrightarrow\,
\mathcal{O}_{\mathbb{P}^3}(1)\,\longrightarrow\, 0.$$ Then $\mathcal{N}$ is stable. The vector bundle
$\mathcal{N}(1)$ is globally generated (see \cite[Theorem 2.1]{SW190}). Further, it can be shown that $\mathcal{N}(1)$ is not ample; see \cite[7.3]{B1}. 

Hence $\mathcal{N}(k)$ is ample if and only if $k \ge 2$ and 
$$\varepsilon(\mathbb{P}^3,\mathcal{N}(k),\,x)\,\geq\, 1$$ for all $x\,\in\, \mathbb{P}^3$ and for all $k\geq 2$.
\end{exm}

\begin{exm}
Next consider the rank two stable vector bundle $E$ on $\mathbb{P}^1\times \mathbb{P}^1$ which
fits in the exact sequence
$$0\,\longrightarrow\, \mathcal{O}_{\mathbb{P}^1\times \mathbb{P}^1}(-1,-1)
\,\longrightarrow\, \mathcal{O}_{\mathbb{P}^1\times \mathbb{P}^1}^{\oplus 3}\,\longrightarrow\, E(1,1)\,\longrightarrow\, 0.$$
The vector bundle $E(1,\,1)$ is globally generated, and hence nef (see \cite[Proposition 3.11]{SW290}).
Hence $$\varepsilon\bigl(\mathbb{P}^1\times \mathbb{P}^1,E(2,\,2),\,x\bigr) \,\geq\, 1$$ for all $x\,\in\, \mathbb{P}^1\times \mathbb{P}^1$.
\end{exm}
 Based on the observations about Fano bundles in this section, the following question may be
more tractable.

\begin{question}\label{qn2}
Is it true that for an ample Fano vector bundle $E$ on $\mathbb{P}^n$,
$$\varepsilon(\mathbb{P}^n,\,E,x) \,\geq\, 1
$$
for all $x\,\in\, \mathbb{P}^n$?
\end{question}

We now show that the observations in this section settle the above question when the rank of $E$ is two. 

\begin{prop}
If $E$ is an ample Fano vector bundle of rank two on $\mathbb{P}^n$,
then $\varepsilon(\mathbb{P}^n,\,E,x) \,\geq\, 1$
for all $x\,\in\, \mathbb{P}^n$. 
\end{prop} 
\begin{proof}
The case $n=1$ is easy: 
every ample vector bundle is a direct sum of ample line bundles and the required inequality clearly holds. 
More generally, the conclusion holds when $E$ is a direct sum of ample line bundles on $\mathbb{P}^n$. 

For $n\,=\,2,\,3$, \cite{SW190,SW290} give a complete
classification of normalized Fano rank two bundles. A vector bundle
$E$ is said to be \textit{normalized} if $c_1(E)$ is either 0 or $-1$.
By twisting any rank 2 vector bundle with a suitable line bundle one may assume that it is normalized. 

For $n=2$, \cite[Theorem]{SW290} shows that 
the only normalized, non-split Fano bundles of rank 2 are the ones studied in 
Examples \ref{ex3.4}, \ref{exm2}, \ref{exm3} and $T_{\mathbb{P}^2}(-2)$, where
$T_{\mathbb{P}^2}$ is the tangent bundle $\mathbb{P}^2$. 
For a discussion on the Seshadri constants of 
the tangent bundle, Example \ref{exm5}. 
In these examples, we showed that the Seshadri constants of 
every ample vector bundle are at least 1. 

For $n=3$, \cite[Theorem 2.1]{SW190} shows that the only 
normalized, non-split Fano bundle of rank 2 is the null correlation bundle $\mathcal{N}$ studied in 
Example \ref{exm4}, where we showed that the Seshadri constants are at least 1 whenever
$\mathcal{N}(k)$ is ample. 

For $n\,\ge\, 4$, \cite[Main Theorem
2.4(1)]{APW94} shows that any Fano vector bundle of rank 2 on $\mathbb{P}^n$ is split. 
\end{proof}

\vspace{2mm}

We now give an example of a vector bundle for which the hypothesis of Proposition \ref{prop1} does
\textit{not} hold, and yet \eqref{iel} holds.

\begin{exm}\label{exm3.9}
Let $Z\,=\,\{p_1,\,\cdots,\, p_5\}$ be five distinct points in $\mathbb{P}^2$ such that no
three of them are collinear. Denote the ideal sheaf of $Z$ by $\mathcal{I}_Z$.
Then there exists a vector bundle $E$ of rank 2 on $\mathbb{P}^2$ which fits in the 
following short exact sequence:
$$
0 \,\longrightarrow\, \mathcal{O}_{\mathbb{P}^2}\,\stackrel{\cdot s}{\longrightarrow}\, E
\,\longrightarrow\, \mathcal{I}_Z \,\longrightarrow\, 0.$$
In other words, there is a section $s\, \in\, H^0(\mathbb{P}^2,\, E)$ whose zero locus is
precisely $Z$. Such vector bundles exist for any finite set of points in $\mathbb{P}^2$; see 
\cite[Section 5.2, Example 1, Page 103]{OSS}. 

We will show that $E(2)$ is not nef while $E(3)$ is ample. 

Let $C$ be the unique conic passing through all the points in $Z$. Restricting the section
$s\, \in\, H^0(\mathbb{P}^2,\, E)$ to $C$, we obtain a short exact sequence: 
$$0 \,\longrightarrow\, \mathcal{O}_{C}(D) \,\longrightarrow\, E\big\vert_C \,\longrightarrow\, \mathcal{O}_C(D')
\,\longrightarrow\, 0,$$ where $\deg(D) =5, \deg(D')=-5$. Tensoring this sequence with
$\mathcal{O}_{\mathbb{P}^2} (2)\big\vert_C$, we obtain the exact sequence
$$ 0 \,\longrightarrow\, \mathcal{O}_{C}(F) \,\longrightarrow\, E(2)\big\vert_C \,\longrightarrow\,
\mathcal{O}_C(F') \,\longrightarrow\, 0,$$
where $\deg(F) = 9, \deg(F') =-1$.
Since $E(2)\big\vert_C$ has a quotient bundle of negative degree, it is not nef.
Consequently, $E(2)$ is not a nef line bundle on $\mathbb{P}^2$. 

The vector bundle $E(3)$ will be shown to be ample using the Seshadri criterion
for ampleness of vector
bundles. We will in fact show that
$\varepsilon(\mathbb{P}^2,E(3),\,x) \,\geq\, 1$ for all $x\,\in \,\mathbb{P}^2$. This
would show that $E(3)$ is ample, by \cite[Theorem 3.11]{FM21}.

The description of Seshadri constant of $E$ at $x$ recalled in \eqref{curves} says 
$$\varepsilon(\mathbb{P}^2, E(3),\,x) \,= \,\,\,\inf\limits_{x\in B\subset X}\frac{\mu_{\min}\bigl(\nu^{\star}(E(3))\bigr)}{\mult_xB},$$
where $B\,\subset\, \mathbb{P}^2$ is any irreducible curve, possibly singular, with its
normalization $\alpha\, :\, B'\, \longrightarrow\, B$, and $\nu \,:\, B'
\,\longrightarrow \, \mathbb{P}^2$ is the composition of $\alpha$
with the inclusion map $B\,\hookrightarrow\, \mathbb{P}^2$ ($\mu_{\min}$ was defined
in Section \ref{sect2}).

Writing $E(3)\,=\, E(2) \otimes \mathcal{O}_{\mathbb{P}^2} (1)$, we have the following equality: 
$$\varepsilon(\mathbb{P}^2,E(3),\,x)\,=\,\,\, \inf\limits_{x\in B\subset X}\Bigl\{\frac{\mu_{\min}\bigl(\nu^{\star}(E(2)\otimes \mathcal{O}(1))\bigr)}{\mult_xB}\Bigr\}.$$
 By \cite[Lemma 3.28]{FM21},
$$\mu_{\min}\bigl(\nu^{\star}(E(2)\otimes \mathcal{O}(1))\bigr)\,=\,
\mu_{\min}(\nu^{\star}(E(2)))+\mu_{\min}(\nu^{\star}(\mathcal{O}(1)))\,=\,
\mu_{\min}(\nu^{\star}(E(2))) + \deg(B).$$
So
$$\varepsilon(\mathbb{P}^2,E(3),\,x)\,=\,\,\, \inf\limits_{x\in B\subset X}
\Bigl\{\frac{\mu_{\min}(\nu^{\star}(E(2)))}{\mult_xB}+\frac{\deg(B)}{\mult_xB}\Bigr\}.$$

We claim that 
for every curve $B \,\subset\, \mathbb{P}^2$ 
containing $x$ which is not equal to the conic $C$, 
one has the inequality $$\mu_{\min}(\nu^{\star}(E(2))) \,\ge\, 0.$$

Since $\deg(B) \,\ge\, \mult_xB$ for all curves $B$ passing through $x$, this shows that the contribution of any curve $B\ne C$ to $\varepsilon(\mathbb{P}^2,E(3),\,x)$ is at least 1. We will also show that the contribution of $C$ to $\varepsilon(\mathbb{P}^2,E(3),\,x)$ is at least 1.

Denote $m_i\,=\, \text{mult}_{p_i}(B)$ for $1 \,\le\, i \,\le\, 5$, and also denote
$d\,=\, \deg(B)$. Then restricting the section $s$ of $E$ to $B$ we obtain the short exact sequence
$$
0\, \longrightarrow\, \mathcal{O}_B(R)\, \longrightarrow\, E(2)\big\vert_B
\, \longrightarrow\, \mathcal{O}_B(R')\, \longrightarrow\, 0,$$
where $\deg(R)=2d+\sum\limits_{i=1}^5 m_i$ and $\deg(R') = 2d-\sum \limits_{i=1}^5m_i$.

Pulling this sequence back to the normalization $B'$ we obtain 
$$0 \, \longrightarrow\, \mathcal{O}_{B'}(P)\, \longrightarrow\, \nu^{\star}(E(2))
\, \longrightarrow\, \mathcal{O}_{B'}(P')\, \longrightarrow\, 0,$$
where $\deg(P) = 2d+\sum\limits_{i=1}^5m_i$ and $\deg(P') = 2d-\sum \limits_{i=1}^5m_i$.

It is easy to see that $$\mu_{\min}(\nu^{\star}(E(2)) \,=\, 2d-\sum m_i.$$

If $B\,=\,C$ is the unique conic passing through $p_1,\, p_2,\, \cdots,\, p_5$, then 
its contribution to the Seshadri constant $\varepsilon(\mathbb{P}^2,E(3),\,x)$ is given by
$\frac{-1}{1}+\frac{2}{1} \,=\, 1$.

It can be shown that $2d\,\ge\, \sum m_i$ for every curve $B \,\neq\, C$ passing through $x$. 
This actually follows easily from B\'ezout's theorem. Indeed, since $B\,\neq\, C$, 
we have $2d \,=\, B\cdot C \,\ge \,\sum m_i$.

This shows that $\varepsilon(\mathbb{P}^2,E(3),\,x) \,\ge\, 1$ for all $x \,\in\, \mathbb{P}^2$, as required.

Further, note that $\varepsilon(\mathbb{P}^2, E,\,x) \,=\, 1$ for any point $x$ belonging to $C$ or to any of
the ten lines which pass through any two of the points $p_1,\,p_2,\, \cdots,\, p_5$. 
\end{exm}

\begin{exm}\label{exm5}
Let $E$ be an equivariant ample vector bundle of rank $r$ on $\mathbb{P}^n$. Recall that there
are only finitely many torus equivariant curves in $\mathbb{P}^n$, say $l_1,\,l_2,\,\cdots,\,l_m$, where
$m\,=\, \frac{n(n+1)}{2}$. Then by \cite[Remark 4.2]{DKS22}, for any point $x\,\in\, \mathbb{P}^n$,
$$\varepsilon(\mathbb{P}^n,\,E,\,x) \,\geq\,\,\, \min\limits_{1\leq i\leq m}\Bigl\{\mu_{\min}(E\vert_{l_i})\Bigr\}.$$

As $E$ is ample, the restriction $E\big\vert_{l_i}$ is also ample. Hence $\varepsilon(E\big\vert_{l_i},\,y)
\,=\, \mu_{\min}(E\big\vert_{l_i})$ for every point $y\,\in\, l_i$. As each $l_i$ is isomorphic to
$\mathbb{P}^1$, we conclude that $\mu_{\min}(E\big\vert_{l_i})\,\geq\, 1$ for every $i$. Thus
$$
\varepsilon(\mathbb{P}^n,\,E,\,x) \,\geq\,\,\, \min\limits_{1\leq i\leq m}\Bigl\{\mu_{\min}
(E\vert_{l_i})\Bigr\} \,\geq\, 1, 
$$
for every $x\,\in\, \mathbb{P}^n$.

For example, consider the tangent bundle $T_{\mathbb{P}^n}\, \longrightarrow\, \mathbb{P}^n$
for $n\,\geq\, 2$. Its restriction to any line $l\,\subset\, \mathbb{P}^n$
is given by 
$$T_{\mathbb{P}^n}\big\vert_l\, =\, \mathcal{O}_{\mathbb{P}^1}(2)\oplus \mathcal{O}_{\mathbb{P}^1}(1)\oplus \cdots \oplus \mathcal{O}_{\mathbb{P}^1}(1),$$
so that $$\mu_{\min}(T_{\mathbb{P}^n}\big\vert_l)\, =\,\, \min\limits_{1\leq i\leq m}\Bigl\{\mu_{\min}(T_{\mathbb{P}^n}\big\vert_{l_i})\Bigr\}
$$
for every line $l\,\subset\, \mathbb{P}^n$.
Hence by \cite[Theorem 4.1]{DKS22} we conclude that for any point $x\,\in \,\mathbb{P}^n$ 
$$\varepsilon(\mathbb{P}^n,\,T_{\mathbb{P}^n},\,x)\, =\,\,\,
\min\limits_{1\leq i\leq m} \Bigl\{\mu_{\min}(T_{\mathbb{P}^n}\big\vert_{l_i})\Bigr\} \,=\, 1.$$

Note that $T_{\mathbb{P}^n}(k)$ is ample if and only if $k \ge 0$. 
\end{exm}

\begin{exm}
More generally, Ein-Lazarsfeld inequality \eqref{iel} holds for torus-equivariant ample vector bundles at 
\textit{torus fixed} points, on: (a)\, toric varieties \cite{HMP10}, (b)\, flag
varieties \cite{BHN20} and (c)\, 
wonderful compactifications as well as Bott-Samelson varieties \cite{BHK22}. In all these cases, 
Seshadri constant at any torus fixed point can be computed by restricting to invariant curves 
which are smooth rational. Hence the restricted bundle is a direct sum of ample line bundles, and its 
Seshadri constant is the minimum degree of a line bundle appearing in the decomposition over all 
invariant curves. Since the degree is a positive integer, it follows that the Seshadri constant
is at least 1.
\end{exm}

\section{Lower bounds for Seshadri constants of ample vector bundles}\label{sect4}

We begin this section by showing that Seshadri constants of ample vector bundles can be 
arbitrarily small. This is an analogue of a similar result of Miranda for line bundles; see 
\cite{EL93}.

\begin{thm}
Fix any positive real number $\delta \,>\,0$ and a positive integer $n$. Then there is a smooth
projective variety $X$ of dimension $n$ and an ample vector bundle $E$ on $X$ such that
$$ \varepsilon(X,E,\,x) \,\,< \,\,\delta$$ for all $x\,\in\, X$.
\end{thm}

\begin{proof}
Let $r$ be a positive integer such that $0\,<\,\frac{1}{r}\,<\,\delta$. Fix a smooth projective curve 
$C$ of genus $g\,\geq\, 2$ and a smooth irreducible projective variety $Y$ of dimension
$n-1\,\geq \,0$. Consider the product $X \,=\, C\times Y$ together with the projections
$$p_1 \,:\, X \,\longrightarrow\, C\ \ \text{ and }\ \ p_2 \,:\, X\,\longrightarrow\, Y.$$

Let $x \,=\, (c,\,y)\,\in\, X$ be an arbitrary point of $X$. Further we choose a very ample
line bundle $L$ on $Y$. Let $V_r$ be a stable vector bundle of rank $r$ on $C$ such that 
$\det(V_r) \,\cong\, \mathcal{O}_C(c)$. The vector $V_r$ is ample because it is stable of
positive degree \cite[p.~520, Corollary 3.10]{Bi}.
Define $E_r \,:= \, p_1^{\star}V_r \otimes p_2^{\star}L$. Take
the smooth curve $B \,:=\, C\times\{y\}\, \stackrel{i}{\hookrightarrow}\, C\times Y\,=\, X$.
Now by definition, $$\varepsilon(X,E_r,\,x) \,\leq\, \mu_{\min}(i^{\star}E_r) \,=
\,\mu(V_r) \,=\, \frac{1}{\rank(V_r)} \,=\, \frac{1}{r} \,<\, \delta.$$
This completes the proof.
\end{proof}

We now prove that Seshadri constants of vector bundles are bounded below by the reciprocal of the 
rank of the bundle, under suitable conditions. This is analogous to the case of vector bundles on 
curves where this lower bound always holds, by \cite[Corollary 3.1]{Hac}. We consider two 
situations where this bound generalizes to the higher rank case and give some examples that 
satisfy the required hypotheses.

\begin{thm}\label{thm2.2}
Let $E$ be an ample vector bundle on a smooth projective variety $X$. Let
\begin{align}\label{seq1}
0 \,=\, E_l \,\subsetneq\, E_{l-1} \,\subsetneq\, E_{l-2}\,\subsetneq\,\cdots\,\subsetneq\,
E_{1} \,\subsetneq\, E_0 \,=\, E
\end{align} 
be the Harder-Narasimhan filtration of $E$ such that $Q_1 = E_0/E_1$ is locally free and $$\mu_{\min}(\nu^{\star}E) \,=\,
\mu(\nu^{\star}Q_1) \,=\, \frac{\det(Q_1)\cdot C}{\rank(Q_1)}$$
for every curve $C\,\subset\, X$, where $\nu\,:\, \widetilde{C}\,\longrightarrow\, X$ is the composition of the
normalization $\widetilde{C}\,\longrightarrow\, C$ with the inclusion map $C\,\hookrightarrow\, X$. Then for any $x\in X$
$$\varepsilon(X,E,\,x)\,=\, \frac{\varepsilon(X,\det(Q_1),\,x)}{\rank(Q_1)}.$$
In particular, if $X$ is a smooth surface, then $$\varepsilon(X,E,\,x) \geq \frac{1}{\rank(E)}$$ for very general points $x\in X$.

In addition, if $E$ is also globally generated vector bundle on the smooth projective variety $X$, then
 $$\varepsilon(X,E,\,x) \geq \frac{1}{\rank(E)}$$ for all points $x\in X$.
\end{thm}

\begin{proof}
By \cite[Corollary 3.1]{FM21}, we have
$$\varepsilon(X,E,\,x) \,=\,\,
\inf\limits_{x\in C \subset X}\Bigl\{\frac{\mu_{\min}(\nu^{\star}E)}{\mult_xC}\Bigr\},$$
where the infimum is taken over all irreducible curves $C\,\subset\, X$ through $x$, and $\nu \,:\,
\widetilde{C}\,\longrightarrow\, X$ is composition of normalization map
$\widetilde{C}\,\longrightarrow\, C$ with the inclusion map $C\,\hookrightarrow\, X$.
By hypothesis, we have $$\mu_{\min}(\nu^{\star}E) \,=\, \mu(\nu^{\star}Q_1)
\,=\, \frac{\det(Q_1)\cdot C}{\rank(Q_1)}.$$
This shows that $$\varepsilon(E,\,x) \,=\,\, \frac{1}{\rank(Q_1)}\inf\limits_{x\in C\subset X}
\Bigl\{\frac{\det(Q_1)\cdot C}{\mult_xC}\Bigr\} \,= \,\frac{1}{\rank(Q_1)}\varepsilon(X,\det(Q_1),\,x).$$
Now the result follows from the Ein-Lazarsfeld inequality for Seshadri constants of ample line
bundles; see \cite{EL93}.
\end{proof}

\begin{exm}
A vector bundle $V$ on an abelian variety $X$ is called \textit{weakly-translation invariant} 
(semi-homogeneous in the sense of Mukai) if for every closed point $x \,\in\, X$, there is a line 
bundle $L_x$ on $X$ depending on $x$ such that $T_x^{\star}(V) \,\simeq\, V \otimes L_x$ for all 
$x\in X$, where $T_x\, :\, X\, \longrightarrow\, X$ is the translation morphism $y\, \longmapsto\, y+x$.

Mukai proved that a semi-homogeneous vector bundle $E$ is Gieseker semistable 
(see \cite[Chapter 1]{H-L} for definition) with respect to some polarization and it has projective 
Chern classes zero, i.e., if $c(E)$ is the total Chern class of $E$, then $c(E) \,=\, \Bigl\{ 1+ 
c_1(E)/r\Bigr\}^r$ (see \cite[p.~260, Theorem 5.8]{Muk78}, \cite[Proposition 6.13]{Muk78}; also 
see \cite[Page 2]{MN84}). Gieseker semistability implies slope semistability (see \cite{H-L}). 
So, in particular, $E$ is slope semistable with $$c_2(\End(E))
\,=\, 2rc_2(E) - (r -1)c_1^2(E) \,=\, 0.$$
Thus by \cite[Theorem 1.2]{BB08} $\nu^*E$ is also semistable for every curve $C\subset X$, where $\nu$ is as in Theorem \ref{thm2.2}. Hence for an ample semi-homogeneous vector bundle $E$ of rank $r$ on an abelian surface $X$, we have $$\varepsilon(X,E,\,x)
\, =\, \frac{\varepsilon(X,\det(E),\,x)}{\rank(E)}\,\geq\, \frac{1}{\rank(E)},$$ for all but countable many points $x\,\in\, X$.
\end{exm}

\begin{exm}
Let $\rho \,:\, X \,=\,\mathbb{P}_C(W)\,\longrightarrow\, C$ be a ruled surface on a smooth projective curve $C$ with
$\mu_{\min}(W) \,=\, \deg(W)$. Take a vector bundle $V$ on $C$, and let $m\,\in\, \mathbb{N}$
be such that $\mu_{\min}(V) \,>\, -m\deg(W)$. So $E \,:=\, \rho^{\star}(V)\otimes 
\mathcal{O}_{\mathbb{P}(W)}(m)$ is an ample vector bundle on $X$ (see \cite[Corollary 10]{MR21}).

Let $$
0\,=\,V_l\,\subsetneq\, V_{l-1}\,\subsetneq\,V_{l-2}\,\subsetneq\,\cdots\,\subsetneq\, V_{1}\,\subsetneq\, V_0 \,=\, V$$
be the Harder-Narasimhan filtration of $V$, and let $Q_i\,:=\,V_{i-1}/V_{i}$, $1\, \leq\, i\, \leq\, l$.
Set $E_i \,:= \,(\rho^{\star}V_i)\otimes \mathcal{O}_{\mathbb{P}(E)}(m)$. 
Then the uniqueness of the Harder-Narasimhan filtration imply that
$$ 0\,=\,E_l\,\subsetneq\, E_{l-1}\,\subsetneq\, E_{l-2}\,\subsetneq\,\cdots\,\subsetneq\, E_{1} \,\subsetneq\, E_0\,=\,E$$
is the Harder-Narasimhan filtration of $E$ with respect to any polarization $H$.

We conclude using \cite[Corollary 10]{MR21} that each $R_i\,:= E_{i-1}/E_i = \,\rho^{\star}(Q_i)\otimes \mathcal{O}_{\mathbb{P}(W)}(m)$ is an ample semistable
bundle on $X$ with $c_2(\End(R_i)) \,=\,0$. Thus $$\varepsilon(X,R_i,\,x) = \frac{1}{\rank(R_i)}\varepsilon(X,\det(R_i),x) \,\,\geq \,\,\frac{1}{\rank(R_i)}$$
for all but countably many points $x\,\in\, X$. So 
inductively we have each $E_i$ is ample, and by \cite[Lemma 3.2]{FM21} we have $$\varepsilon(X,E_i,\,x) \,\geq \,
\min\Bigl\{\varepsilon(X,E_{i-1},\,x),\, \varepsilon(X,R_i,\,x)\Big\} \,\geq\, \frac{1}{\rank(E_i)}$$
for all but countably many points $x\,\in\, X$.
In particular, using Theorem \ref{thm2.2},
$$\varepsilon(X,E,\,x) \,= \,\frac{\varepsilon(X,\det(R_1),\,x)}{\rank(R_1)} \,\geq\, \frac{1}{\rank(R_1)}\,\geq \,\frac{1}{\rank(E)}$$
for all but countably many points $x\,\in\, X$.
\end{exm}

\begin{exm}
Let $B$ and $D$ be any two smooth projective curves. Consider the product $X\,=\,B\times D$ together with the projections $p_1\,:
\, X\,\longrightarrow\,C$ and $p_2 \,:\, X\,\longrightarrow\, D$.
 
Let $V$ be an ample vector bundle on $B$ and $L$ a very ample line bundle on $D$. Then $E\,:=\, p_1^{\star}V\otimes p_2^{\star}L$
is an ample vector bundle on $X$. Let 
\begin{align*}
0 \,=\, V_{l} \,\subsetneq\, V_{l-1}\,\subsetneq\, V_{l-2}\,\subsetneq\, \cdots\,\subsetneq\, V_1\,\subsetneq\, V_0 \,=\, V
\end{align*}
be the Harder-Narasimhan filtration of $V$. Then 
\begin{align*}
0 \,=\, E_{l} \,\subsetneq\, E_{l-1}\,\subsetneq\, E_{l-2} \,\subsetneq\, \cdots \,\subsetneq\, E_1\,\subsetneq\, E_0 \,=\, E
\end{align*}
is the Harder-Narasimhan filtration of $E$, where $E_i = p_1^{\star}V_i\otimes p_2^{\star}L$ for each $i$. Note that each
$Q_i \,:= \,E_{i-1}/E_i$ is semistable with $c_2(\End(Q_i))\,=\,0$. In particular $\mu_{\min}(\nu^*E) = \mu(\nu^*Q_1)$ for every curve $C\subset X$, where $\nu$ is as in Theorem \ref{thm2.2}.

Therefore $$\varepsilon(X,E,\,x) \,= \,\frac{\varepsilon(X,\det(Q_1),\,x)}{\rank(Q_1)} \,\geq\, \frac{1}{\rank(Q_1)}
\,\geq \,\frac{1}{\rank(E)}$$
for all but countably many points $x\,\in\, X$.
\end{exm}

\begin{exm}
Let $\pi \,:\, X \,:=\, \mathbb{P}_C(\mathcal{F})\,\longrightarrow\, C$ be a projective bundle over an elliptic curve $C$, where
$\mathcal F$ is an indecomposable vector bundle of rank $n$ and degree $1$ on $C$. Let
$\mathcal{O}_{\mathbb{P}_C(\mathcal{F})}(1)$ be the tautological line bundle, and set 
$E\,:=\, \pi^{\star}\mathcal{E}\otimes \mathcal{O}_{\mathbb{P}_C(\mathcal{F})}(1)$, where $\mathcal{E}$ is an indecomposable vector
bundle of rank $n$ and degree $1$. Then $E$ is an ample and globally generated vector bundle (see \cite[Theorem 1.1, (4)]{N98}). Further
$\mathcal{E}$ is an indecomposable bundle, and hence semistable. As a result, $E$ is a semistable bundle with $c_2(\End(E))\,=\,0$. Thus
$$\varepsilon(X,E,\,x) \,=\, \frac{\varepsilon(X,\det(E),\,x)}{\rank(E)}\,\geq\, \frac{1}{\rank(E)}$$
for all $x\,\in\, X.$
\end{exm}

\section{ample vector bundles on Hirzebruch Surfaces}\label{sect5}

In this section we consider rank 2 vector bundles on Hirzebruch surfaces. Such bundles with small values of $c_2$ have been classified in the 
literature, see \cite{I} for instance.

In the following proposition, we list the indecomposable rank 2 ample bundles on Hirzebruch 
surfaces with small $c_2$. For all these vector bundles it will be shown that the Seshadri 
constant is at least 1 at all points.

\begin{prop}[{\cite[Theorem 2.9]{I}}]\label{hirzebruch}
Let $\pi \,:\, \mathbb{F}_e \,=\, \mathbb{P}(\mathcal{O}_{\mathbb{P}^1}\oplus \mathcal{O}_{\mathbb{P}^1}(-e)) \,\longrightarrow\,
\mathbb{P}^1$ be a Hirzebruch surface and $E\, \longrightarrow\, \mathbb{F}_e$ an indecomposable ample vector bundle of rank two.
Denote the normalized section of $\pi$ by $\sigma$ and a fiber of $\pi$ by 
$f$. Then $c_2(E)\,\geq\, e+5$. Moreover, if $c_2(E)\,\leq\, e+6$, then one of the following cases 
occur:

\textbf{Case 1:}\, $E$ is an ample indecomposable rank 2 bundle on $\mathbb{F}_{1}$ with $c_2(E)\,=\, 6$ 
which sits in the following exact sequence: $$0\,\longrightarrow\, 2\sigma+2f\,\longrightarrow\, E 
\,\longrightarrow\, \sigma+3f \,\longrightarrow \,0.$$
In this case the Nef cone of $\mathbb{F}_e$ is generated by $\{\sigma+f,\,f\}$.

\textbf{Case 2:}\, $E$ is an ample indecomposable rank 2 bundle on $\mathbb{F}_{0}$ with $c_2(E)\,=\, 
6$ which sits in the following exact sequence: $$0\,\longrightarrow\, 2\sigma\,\longrightarrow\, E 
\,\longrightarrow \,\sigma+3f \,\longrightarrow\, 0.$$
In this case the Nef cone is generated by $\{\sigma,\, f\}$.

\textbf{Case 3:}\, $E$ an ample indecomposable rank 2 bundle on $\mathbb{F}_{1}$ with $c_2(E)\,=\, 
7$ which sits in the following exact sequence: $$0\,\longrightarrow\, 2\sigma+f\,\longrightarrow\, E 
\,\longrightarrow\, \sigma+4f\,\longrightarrow\, 0.$$
In this case the Nef cone of $\mathbb{F}_e$ is generated by $\{\sigma+f,\,f\}.$

\textbf{Case 4:}\, $E$ is an ample indecomposable rank 2 bundle on $\mathbb{F}_{2}$ with $c_2(E)
\,=\, 8$ which sits in the following exact sequence:
$$0\,\longrightarrow\, 2\sigma+4f\,\longrightarrow\, E \,\longrightarrow\, \sigma+4f
\,\longrightarrow\, 0.$$
In this case the Nef cone of $\mathbb{F}_e$ is generated by $\{\sigma+2f,\,f\}$.
\end{prop}

\begin{thm}\label{hirzebruch2}
Let $\pi \,: \,\mathbb{F}_e \,=\, \mathbb{P}(\mathcal{O}_{\mathbb{P}^1}\oplus \mathcal{O}_{\mathbb{P}^1}(-e))\,\longrightarrow\,\mathbb{P}^1$
be a Hirzebruch surface, and let $E$ be one of the four indecomposable ample vector bundle of rank 2 on $\mathbb{F}_e$
listed in Proposition \ref{hirzebruch}. Then $$\varepsilon(\mathbb{F}_e, E,\,x) \,\ge\, 1$$ for all $x\,\in\, \mathbb{F}_e$.
\end{thm}

\begin{proof}
We will use \cite[Corollary 3.21]{FM21} which allows us to compute the Seshadri constants 
$\varepsilon(\mathbb{F}_e,E,\,x)$ by restricting $E$ to curves on $\mathbb{F}_e$; see \eqref{curves}.
To prove the theorem it suffices to show that 
\begin{eqnarray}\label{ineq}
\frac{\mu_{\min}(\nu^{\star}E)}{\mult_xC} \,\ge\, 1
\end{eqnarray}
for every irreducible curve $C \,\subset\, \mathbb{F}_e$ containing $x$. Suppose
that $C$ is numerically equivalent to $a\sigma+bf$ for some integers $a,\, b$. 

We will prove \eqref{ineq} separately for each of the four cases in Proposition \ref{hirzebruch}. Take 
$x\,\in\,\mathbb{F}_e$.

\textbf{Case 1:}\, Here $e\,=\,1$. Let $C \,\subset\, \mathbb{F}_1$ be an irreducible curve containing $x$. 
Restricting the defining short exact sequence to $C$, we obtain
\begin{eqnarray}\label{exact1}
0\,\longrightarrow\, \mathcal{O}_C(D) \,\longrightarrow\, E\big\vert_C\,\longrightarrow\, \mathcal{O}_C(D')\,\longrightarrow\, 0,
\end{eqnarray}
where $\deg(D) = 2b$ and $\deg(D') = 2a+b$.

Pulling the above sequence back to the normalization $\alpha\, :\, \widetilde{C}\, \longrightarrow\, C$, we obtain
an exact sequence
$$0\,\longrightarrow\, \nu^{\star}(2\sigma+2f)\,\longrightarrow\, \nu^{\star} E \,\longrightarrow\, \nu^{\star}(\sigma+3f )\,\longrightarrow\, 0,$$
where $\nu$ is the composition of $\alpha$ with the inclusion map $C\, \hookrightarrow\, {\mathbb F}_1$.

Then $$\text{degree}(\nu^{\star}(2\sigma+2f)) \,=\, \text{degree}(2\sigma+2f\big\vert_C) \,=\, 2b$$
and $$\text{degree}(\nu^{\star}(\sigma+3f)) \,=\, \text{degree}(\sigma+3f\big\vert_C) \,=\, 2a+b.$$

By \cite[Corollary V.2.18(b)]{Har77}, the only possibilities are $a=1,b=0$, $a=0,b=1$ or $b \ge a > 0$. Further, 
intersecting $C$ with the fiber through $x$, we obtain that $$a \,= \,C\cdot f \,\ge\, {\mult_x C}.$$ 

First, suppose that $a > 0$ and $b> 0$. Then 
using Lemma \ref{mu-min}, we obtain the following:
$$\mu_{\min}(\nu^{\star}E) \,\ge\, \min\{2b,\,2a+b\} \,>\, b \,\ge\, a \,\ge\, {\mult_x C}.$$

If $a\,=\,0$ and $b\,=\,1$, then we claim that $C$ is smooth. 
Indeed, the linear system $\sigma+2f$ is very ample (see \cite[Corollary 
V.2.18]{Har77}) and hence it contains a nonsingular curve $B$ through $x$. Then B\'ezout's theorem gives that
$$C\cdot B \,=\, 1 \,\ge\, {\mult_x C}.$$

By the above calculation, $\mu_{\min}(\nu^{\star}E)\,\ge\,\min\{2,\,1\} \,=\, 1 \,=\,{\mult_x C}$. 

If $a\,=\,1$ and $b\,=\,0$, then $C\,\cong\, \mathbb{P}^1$. In this case,
$E\big\vert_C$ is an ample and globally generated vector bundle of rank 2 on $C$. So 
$$\varepsilon(C, E\big\vert_C,\, x) \,=\, \mu_{\min}(E\big\vert_C)\, \geq \,1 \,=\,{\mult_x C}.$$

We conclude that\eqref{ineq} holds in every case. 

\textbf{Case 2:}\, Here $e\,=\,0$. Let $C \,\subset\, \mathbb{F}_0$ be an irreducible curve containing $x$. 
Restricting the defining short exact sequence to $C$, we obtain
\begin{eqnarray}\label{exact2}
0\,\longrightarrow\, \mathcal{O}_C(F)\,\longrightarrow\, E\big\vert_C \,\longrightarrow\, \mathcal{O}_C(F')\,\longrightarrow\, 0,
\end{eqnarray}
where $\deg(F) \,= \,2b$ and $\deg(F') \,=\, 3a+b$.

If $C \,= \,\sigma$, then $C$ is smooth, and ${\mult_xC}\, = \,1$. Then \eqref{exact2} becomes
$$0\,\longrightarrow\, \mathcal{O}_C \,\longrightarrow\, E\big\vert_C \,\longrightarrow\, \mathcal{O}_C(3)\,\longrightarrow\, 0.$$
Since $E$ is ample, we have $E\big\vert_C$ is ample and globally generated. The proof can be completed exactly as done in the analogous situation in Case 1.

If $C$ is numerically equal to $f$, the argument is identical to the one in Case 1. 

Now assume that $C$ is different from $\sigma,\, f$. Then $a\,>\, 0,\,\, b \,>\, 0$ by \cite[Corollary V.2.18(b)]{Har77}.

Note that through every point $x\, \in\, \mathbb{F}_0$ there are curves numerically equivalent to $\sigma$ and $f$. 
Intersecting $C$ with these curves and using B\'ezout, we conclude that $a \,\ge \, {\mult_x C}$ and 
$b\,\ge\,{\mult_x C}$. 

As above, using Lemma \ref{mu-min} for the exact sequence \eqref{exact1} we obtain 
$$\mu_{\min}(\nu^{\star}E) \,\ge\, \min\{2b,\,3a+b\} \,\ge \,{\mult_x C}.$$
Hence \eqref{ineq} holds. 

\textbf{Case 3:}\, Here $e\,=\,1$. Let $C \,\subset\, \mathbb{F}_1$ be an irreducible curve containing $x$. 
Restricting the defining short exact sequence to $C$ we obtain
\begin{eqnarray}\label{exact3}
0\,\longrightarrow\, \mathcal{O}_C(G) \,\longrightarrow\, E\big\vert_C \,\longrightarrow\, \mathcal{O}_C(G')\,\longrightarrow\, 0,
\end{eqnarray}
where $\deg(G) \,=\, 2b-a$ and $\deg(G') \,=\, 3a+b$.

If $C \,=\, \sigma$, then $C$ is smooth, and ${\mult_xC} \,=\, 1$. Then \eqref{exact3} becomes
$$0\,\longrightarrow\, \mathcal{O}_C(-1) \,\longrightarrow\, E\big\vert_C \,\longrightarrow\, \mathcal{O}_C(3)\,\longrightarrow\, 0.$$
Since $E$ is ample, we have $E\big\vert_C$ is ample and globally generated.
The proof can be completed exactly as in the analogous situation in Case 1.

If $C$ is numerically equal to $f$, the argument is identical to the one in Case 1. 

Now suppose $C$ is different from $\sigma,\, f$. Then, by \cite[Corollary V.2.20(a)]{Har77}, $a\,>\, 0$ and $b \,> \,a$. 
Intersecting $C$ with a fiber through $x$, we obtain that $a \,=\, C\cdot f \,\ge\, {\mult_x C}$. 

We conclude exactly as done in Case 1.

\textbf{Case 4}: The argument is identical to the above cases. So we omit its proof. 
\end{proof}

\begin{xrem}
Let $E$ be an equivariant ample vector bundle of rank $r$ on a Hirzebruch surface $X$. There are certain special torus-invariant curves $D_1$,\, $D_2$,\, $D_1'$ and $D_2'$ in $X$;
see \cite[Subsection 2.3 and Section 5]{DKS22} for more details about $D_1$,\, $D_2$,\, $D_1'$ and $D_2'$.

Set $$\mu_1 \,= \,\mu_{\min}(E\vert_{D_1}), \mu_2 \,=\, \mu_{\min}(E\vert_{D_2}), \mu_1' \,=\, \mu_{\min}(E\vert_{D_1'}) \hspace{2mm} \text{and} \hspace{2mm} \mu_2'
\,=\, \mu_{\min}(E\vert_{D_2'}).$$

Further assume that $E$ satisfies the following property as in \cite[Theorem 5.1]{DKS22}:
$$\mu_1 = \mu_{\min}(E\vert_{D_1}) \,=\, \mu_{\min}(E\vert_{D'_1}) = \mu_1'$$
Then for any $x\,\in\, X$, by \cite[Theorem 5.1]{DKS22} 
the Seshadri constants $\varepsilon(X,E,\,x)$ satisfy the following inequalities :
\begin{center}
$\varepsilon(X,E,\,x) \,\geq\, \min\bigl\{\mu_1,\,\mu_2,\,\mu_2'\bigr\}$, if $x\,\in\, \Gamma_2$
\end{center}
and 
\begin{center}
$\varepsilon(X,E,\,x) \,\geq\, \min\bigl\{\mu_1,\,\mu_2\bigr\}$, if $x\,\notin\, \Gamma_2$.
\end{center}

Here $\Gamma_2$ is a certain curve on $X$; see \cite[Proposition 2.4]{DKS22} for a description of $\Gamma_2$.

Since the invariant curves are isomorphic to $\mathbb{P}^1$, in both of the above cases we have
$$\varepsilon(X,E,\,x) \,\geq\, 1,\, \text{ for all } x\,\in\, X.$$
\end{xrem}

\section*{Acknowledgements}
We are grateful to the referee for a careful reading and the many suggestions which improved the paper. The present
proof of Case 1 in Theorem \ref{hirzebruch2}, which is shorter than the earlier one, is due to the referee. We thank G. V.
Ravindra for several useful discussions. The second and the third authors are partially supported by a grant from
Infosys Foundation. The third author is supported financially by SERB-NPDF fellowship (File no: PDF/2021/00028).

\end{document}